\documentclass[12pt,a4paper]{amsart}
\usepackage[english]{babel}
\usepackage[T1]{fontenc}
\usepackage[latin1]{inputenc}
\usepackage{graphicx}
\usepackage{amsmath,amssymb,amsthm,mathrsfs,txfonts}
\usepackage[foot]{amsaddr}
\usepackage{soul}
\usepackage{array, multirow}
\usepackage{stmaryrd}
\usepackage{mathrsfs}
\usepackage{color}
\usepackage{hyperref}
\usepackage{enumerate}
\usepackage[all]{xy}
\usepackage[left=2.5cm, right=2.5cm, top=3cm, bottom=3cm]{geometry}

\theoremstyle{plain}
\usepackage{thmtools}
\usepackage{thm-restate}
\usepackage{hyperref}
\usepackage{cleveref}
\newtheorem{theo}{Theorem}[section]
\newtheorem{lemme}[theo]{Lemma}
\newtheorem{prop}[theo]{Proposition}
\newtheorem{coro}[theo]{Corollary}

\theoremstyle{definition}
\newtheorem{defn}[theo]{Definition}

\theoremstyle{remark}
\newtheorem{rema}[theo]{Remark}

\def\Sp{\mathop{\rm Sp}\nolimits}

\def\M{\mathop{\rm M}\nolimits}
\def\det{\mathop{\rm det}\nolimits}
\def\O{\mathop{\rm O}\nolimits}
\def\GL{\mathop{\rm GL}\nolimits}	
\def\End{\mathop{\rm End}\nolimits}\def\Sp{\mathop{\rm Sp}\nolimits}
\def\Hom{\mathop{\rm Hom}\nolimits}

\def\Ker{\mathop{\rm Ker}\nolimits}

\def\det{\mathop{\rm det}\nolimits}
\def\U{\mathop{\rm U}\nolimits}
\def\G{\mathop{\rm G}\nolimits}
\def\K{\mathop{\mathbb{K}}\nolimits}

\def\S{\mathop{\rm S}\nolimits}

\def\T{\mathop{\rm T}\nolimits}
\def\dim{\mathop{\rm dim}\nolimits}

\def\M{\mathop{\rm M}\nolimits}
\def\C{\mathop{\mathbb{C}}\nolimits}

\def\sgn{\mathop{\rm sgn}\nolimits}
\def\diag{\mathop{\rm diag}\nolimits}

\def\pr{\mathop{\rm pr}\nolimits}

\def\Re{\mathop{\rm Re}\nolimits}

\def\Pin{\mathop{\rm Pin}\nolimits}

\def\Id{\mathop{\rm Id}\nolimits}
\def\H{\mathop{\mathbb{H}}\nolimits}

\def\F{\mathop{\rm F}\nolimits}

\def\T{\mathop{\rm T}\nolimits}
\def\C{\mathop{\rm C}\nolimits}
\def\E{\mathop{\rm E}\nolimits}
\def\F{\mathop{\rm F}\nolimits}
\def\Id{\mathop{\rm Id}\nolimits}

\def\Pin{\mathop{\rm Pin}\nolimits}

\def\Cliff{\mathop{\rm Cliff}\nolimits}
\def\Hom{\mathop{\rm Hom}\nolimits}
\def\SO{\mathop{\rm SO}\nolimits}
\def\Spin{\mathop{\rm Spin}\nolimits}

\def\U{\mathop{\rm U}\nolimits}
\def\K{\mathop{\rm K}\nolimits}

\def\S{\mathop{\rm S}\nolimits}

\def\Comm{\mathop{\rm Comm}\nolimits}

\def\C{\mathop{\rm C}\nolimits}
\def\D{\mathop{\rm D}\nolimits}
\def\H{\mathop{\rm H}\nolimits}

\def\Fix{\mathop{\rm Fix}\nolimits}

\def\Id{\mathop{\rm Id}\nolimits}
\def\V{\mathop{\rm V}\nolimits}
\def\W{\mathop{\rm W}\nolimits}

\def\b{\mathop{\rm b}\nolimits}
\def\irr{\mathop{\rm irr}\nolimits}
\def\Ex{\mathop{\rm Ex}\nolimits}
\def\D{\mathop{\rm D}\nolimits}
\def\P{\mathop{\rm P}\nolimits}

\def\R{\mathbb{R}}
\def\tilde{\widetilde}

\title{Dual pairs in the Pin-group and duality for the corresponding spinorial representation}

\author[1]{Cl\'ement Gu\'erin}
\address[1]{Universit\'e du Luxembourg \\ Campus Belval \\ Maison du Nombre \\ 6 Avenue de la Fonte \\ L-4364 Esch-sur-Alzette \\ Luxembourg}
\email[1]{clement.guerin@uni.lu}
\author[2]{Gang Liu}
\address[2]{Universit\'e de Lorraine \\ Institut Elie Cartan de Lorraine \\ 3 rue Augustin Fresnel \\ 57073 Metz, France}
\email[2]{gang.liu@univ-lorraine.fr}
\author[3]{Allan Merino}
\address[3]{ Department of Mathematics \\ National University of Singapore \\ Block S17 \\ 10, Lower Kent Ridge Road \\ Singapore 119076 \\ Republic of Singapore}
\email[3]{matafm@nus.edu.sg}

\date{}
	
\begin{document}

\maketitle

\begin{abstract}

In this paper, we give a complete picture of Howe correspondence for the setting ($\O(\E, \b), \Pin(\E, \b), \Pi$), where $\O(\E, \b)$ is an orthogonal group (real or complex), $\Pin(\E, \b)$ is the two-fold Pin-covering of $\O(\E, \b)$, and $\Pi$ is the spinorial representation of $\Pin(\E, \b)$. More precisely, for a dual pair ($\G, \G'$) in $\O(\E, \b)$, we determine explicitly the nature of its preimages $(\tilde{\G}, \tilde{\G'})$ in $\Pin(\E, \b)$, and prove that apart from some exceptions, $(\tilde{\G}, \tilde{\G'})$ is always a dual pair in $\Pin(\E, \b)$; then we establish the Howe correspondence for $\Pi$ with respect to $(\tilde{\G}, \tilde{\G'})$.

\end{abstract}

\tableofcontents

\section{Introduction}

The first duality phenomenon has been discovered by H. Weyl who pointed out a correspondence between some irreducible finite dimensional representations of the general linear group $\GL(\V)$ and the symmetric group $\mathscr{S}_{k}$ where $V$ is a finite dimensional vector space over $\mathbb{C}$. Indeed, considering the joint action of $\GL(\V)$ and $\mathscr{S}_{d}$ on the space $\V^{\otimes d}$, we get the following decomposition:
\begin{equation*}
\V^{\otimes d} = \bigoplus\limits_{(\V_{\lambda}, \lambda) \in \widehat{\GL(\V)}_{\pi}} \lambda \otimes \sigma^{\lambda},
\end{equation*}
where $\widehat{\GL(\V)}_{\pi}$ is the set of equivalence classes of irreducible finite dimensional representations of $\GL(\V)$ such that $\Hom_{\GL(\V)}(\V_{\lambda}, \V^{\otimes d}) \neq \{0\}$ and $\sigma_{\lambda}$ is an irreducible representation of $\mathscr{S}_{d}$. The key point in this duality is that the commutator in $\End(V^{\otimes d})$ of the algebra generated by $\GL(\V)$ is precisely the algebra generated by $\mathscr{S}_{d}$.

\bigskip

Later, R. Howe \cite{HOW4} proved a similar phenomenon involving some particular subgroups of the symplectic group. Let $(\W, \left<\cdot, \cdot\right>)$ be a real symplectic space and $(\G, \G')$ be an irreducible reductive dual pair in the corresponding symplectic group $\Sp(W)$. We denote by $\widetilde{\Sp(\W)}$ the metaplectic group and by $(\omega, \mathscr{H})$ the metaplectic representation. One can prove that the preimages $\tilde{\G}$ and $\tilde{\G}'$ of $\G$ and $\G'$ respectively form a dual pair in the metaplectic group $\widetilde{\Sp(\W)}$. We denote by $\mathscr{R}(\tilde{\G}, \omega)$ the set of infinitesimal classes of irreducible admissible representations of $\tilde{\G}$ which can be realised as a quotient of $\mathscr{H}^{\infty}$ by a closed $\omega^{\infty}(\tilde{\G})$-invariant subspace. R. Howe proved that we have a one-to-one correspondence between $\mathscr{R}(\tilde{\G}, \omega)$ and $\mathscr{R}(\tilde{\G}', \omega)$, whose graph is $\mathscr{R}(\tilde{\G}\tilde{\G}', \omega)$. This correspondence is known as Howe correspondence or Theta correspondence.

\bigskip

We can define in a similar way the notion of (reductive) dual pairs for more general groups, among which the orthogonal groups. It is noted that the complete list of reductive dual pairs of  orthogonal groups is already known (\cite{HOW1} and \cite{SCH}) Let $\E$ be a vector space over $\mathbb{K} = \mathbb{R}$ or $\mathbb{C}$ endowed with a non-degenerate symmetric bilinear form $\b$ and $\O(\E, \b)$ be the corresponding group of isometries. We denote by $\Pin(\E, \b)$ the $\Pin$-cover of $\O(\E, \b)$ (Section \ref{SectionPreliminaries}, Equation \eqref{DefinitionPinGroup}). As pointed out by Slupinski in \cite{SLU}, one of the difference with the symplectic case is that for a dual pair $(\G, \G')$ in $\O(\E, \b)$, the preimages $\tilde{\G}$ and $\tilde{\G'}$ do not commute in $\Pin(\E, \b)$ in general. One example is given in Slupinski's paper \cite{SLU}. In his paper, he studied the dual pair $(\O(\V), \O(\W))$ in $\O(\V \otimes \W)$ (when the field $\mathbb{K}$ over which the vector spaces are defined satisfies $\mathbb{K}^{*}/(\mathbb{K}^{*})^{2} \approx \mathbb{Z}_{2}$ and $-1 \notin (\mathbb{K}^{*})^{2})$. In particular, he proved that if one of the space $\V$ or $\W$ is even dimensional (and $\dim_{\mathbb{K}}(\V) \neq 2)$, the preimages $\widetilde{\O(\V)}$ and $\widetilde{\O(\W)}$ do not commute in the $\Pin$-group (nevertheless, he proved that $(\widetilde{\SO(\V)}, \widetilde{\O(\W)})$ is actually a dual pair in $\Pin(\V \otimes \W)$). 
Moreover, for the dual pairs he obtained in the $\Pin$-group, he proved that the spinorial representation $(\Pi, \S)$ of $\Pin(\E, \b)$ sets up a Howe correspondence, i.e.
\begin{equation*}
\Pi = \bigoplus\limits_{\lambda} \lambda \otimes \theta(\lambda),
\end{equation*}
where $\lambda$ runs over the irreducible finite dimensional representations of $\tilde{\G}$ such that $\Hom_{\tilde{\G}}(\lambda, \Pi) \neq \{0\}$ and $\theta(\lambda)$ is an irreducible representation of $\tilde{\G}'$. Hints of this duality were given by R. Howe in \cite{HOW3}.

Let $(\G, \G')$ be a dual pair in $\O(\E, \b)$ and let $(\tilde{\G}, \tilde{\G'})$ be the preimage of $(\G, \G')$ in $\Pin(\E, \b)$. The goal of this paper is twofold. Firstly, we want to determine the nature of  the covering pair $(\tilde{\G}, \tilde{\G'})$, more precisely, we want to know when $(\tilde{\G}, \tilde{\G'})$ is a dual pair in  $\Pin(\E, \b)$, and we also want to determine the isomorphism classes of $\tilde{\G}$ and  $\tilde{\G'}$. Secondly, in the case where$(\tilde{\G}, \tilde{\G'})$ is a dual pair in  $\Pin(\E, \b)$,  we want to prove that the spinorial representation of  $\Pin(\E, \b)$ sets up a Howe correspondance for $(\tilde{\G}, \tilde{\G'})$.

\bigskip

The classification of the dual pairs in the orthogonal group (see Corollary \ref{ClassificationDP}) can be found in \cite{HOW1}. One can also check the paper of Schmidt \cite{SCH} who classified the irreducible dual pairs in the classical groups.  Concerning the dual pairs in $\Pin(\E, \b)$, the method we use in this paper is an adaptation of the one developed by Slupinski in \cite{SLU}. 

We have the following result (with respect to the list given in Corollary \ref{ClassificationDP}).

\begin{restatable}{theo}{dualift}
\label{prop:dualift}
Let $\E$ be a real vector space endowed with a non-degenerate, symmetric, bilinear form $\b$ and $(\G, \G')$ be an irreducible dual pair in $\O(\E, \b)$ which is not real orthogonal (treated in \cite{SLU}). Then its lift $(\tilde{\G},\tilde{\G}')$ in $\Pin(\E, \b)$ is an irreducible dual pair.

\end{restatable}

We get the same results if the space $\E$ is complex (see Remark \ref{CoverComplexPairs}). Furthermore, in Section \ref{SectionDualPairs} we determine the nature of the double covers $\tilde{\G}$ and $\tilde{\G'}$.

\begin{restatable}{prop}{classlift}
\label{theo:classlift}

Let $(\G, \G')$ be an irreducible dual pair in $\O(\E, \b)$ which is not real orthogonal. If $\G, \G'$ are not unitary groups then their lifts in $\Pin(\E, \b)$ are trivial extensions (i.e. direct products). 

If $\G = \U(p_1,q_1)$ and $\G' = \U(p_2,q_2)$ then the lift of $\G$ in $\Pin(2(p_1p_2+q_1q_2), 2(p_1q_2+q_1p_2), \R)$ is isomorphic to 
\[\left\lbrace\begin{array}{ll}\Lambda(p_1,q_1)&\text{if $p_2+q_2$ is odd}\\ \G \times \{\pm 1\}&\text{ otherwise}\end{array}\right.\]

\end{restatable}

Here, $\Lambda(p,q)$ denotes the unique lift of $\U(p,q)$ which restricts to a non-trivial lift above $\U(p)$ and $\U(q)$. The same kind of result is proved for dual pairs in the complex orthogonal group (see Section \ref{SectionDualPairs}, Remark \ref{CoverComplexPairs})

\bigskip 

For the dual pairs found in $\Pin(\E, \b)$ in Section \ref{SectionDualPairs}, we also prove that the spinorial representation sets up a Howe correspondence. 
In Section \ref{SectionDualityPin}, we prove:

\begin{restatable}{theo}{theodual}
\label{theo:dual}
Let $(\tilde{\G}, \tilde{\G'})$ be one of the dual pairs in $\Pin(\E, \b)$ we found in Section \ref{SectionDualPairs}. Then the spinorial representation sets up a Howe correspondance for $(\tilde{\G},\tilde{\G}')$.
\end{restatable}

The main point is to use the double commutant theorem \cite[Section~4.1]{GOD}. It translates the question of the Howe correspondence to a computation of invariants. In this computation, we complexify the groups to use classical results of Howe on invariants of complex groups (\cite{HOW2}).

\bigskip

\textbf{Acknowledgments} \thinspace \thinspace Cl\'ement Gu\'erin was supported by the FNR grant OPEN/16/11405402 and Allan Merino was supported by the grant R-146-000-261-114 (Reductive dual pair correspondences and supercuspidal representations).

\bigskip

\section{Preliminaries}

\label{SectionPreliminaries}

Let $\E$ be a vector space over $\mathbb{K} = \mathbb{R}$ or $\mathbb{C}$ endowed with a non-degenerate, symmetric, bilinear form $\b$, and denote by $\O(\E, \b)$ the corresponding group of isometries
\begin{equation*}
\O(\E, \b) = \left\{g \in \GL(\E), \b(g(u), g(v)) = \b(u, v) \thinspace (\forall u, v \in \E)\right\}.
\end{equation*}

A pair of subgroups $(\G, \G')$ in $\O(\E, \b)$ is called a \underline{dual pair} if $Z_{\O(\E, \b)}(\G')=\G$ and $Z_{\O(\E, \b)}(\G)=\G'$. A dual pair $(\G, \G')$ is called \underline{reductive} if $\G$ and $\G'$ act reductively on $\E$ as subgroups of $\GL(\E)$. 

If we can find an orthogonal decomposition of $\E = \E_{1} \oplus \E_{2}$ where both subspaces are $\G.\G'$-invariant, the restrictions $\G_{|_{\E_{i}}}$ and $\G'_{|_{\E_{i}}}$ are well-defined and $(\G_{|_{\E_{i}}}, \G'_{|_{\E_{i}}})$ is a dual pair in $\O(\E_{i}, \b_{i})$, $i = 1, 2$, where $\b_{i} = \b_{|_{\E_{i} \times \E_{i}}}$. In this case, we say that $(\G, \G')$ is the direct sum of $(\G_{|_{\E_{1}}}, \G'_{|_{\E_{1}}})$ and $(\G_{|_{\E_{2}}}, \G'_{|_{\E_{2}}})$. If no such decomposition exist, the dual pair is said \underline{irreducible}.

As proved by Howe in \cite[Section~6]{HOW1}, every dual reductive dual pair $(\G, \G')$ in $\O(\E, \b)$ is a direct sum of irreducible dual pairs. Indeed, if we consider the decomposition of $\E$ as a direct sum of $\G$-isotypic components
\begin{equation*}
\E = \bigoplus\limits_{i=1}^{n} \E_{i},
\end{equation*}
one can prove that the restriction of $\b$ to $\E_{i}$ is either zero or non-degenerate. In particular, if $\E_{i}$ is isotropic, then there exists a unique $\E_{j}$ such that the restriction of $\b$ to $\E_{i} \times \E_{j}$ is non-degenerate. Then, we have
\begin{equation*}
\E = \bigoplus\limits_{k=1}^{m} \widetilde{\E(k)}
\end{equation*}
where $\widetilde{\E(k)}$ is either a non-degenerate $\G$-isotypic component or a direct sum of two non-degenerate $\G$-isotypic components such that the pairing is non-degenerate. The group $\G'$ acts on $\widetilde{\E(k)}$, and the sum is orthogonal. In particular, $(\G_{|_{\widetilde{\E_{k}}}}, \G'_{|_{\widetilde{\E_{k}}}})$ is an irreducible reductive dual pair in $\O(\widetilde{\E_{k}}, \tilde{\b_{k}})$.

According to the previous paragraph, we have two differents kinds of irreducible dual pairs: the dual pair is said of \underline{type I} if the action of $\G.\G'$ on $\E$ is irreducible, and of \underline{type II} if $\E = \E_{1} \oplus \E_{2}$ is a direct sum of maximal isotropic subspaces invariant by $\G.\G'$. A classification of such dual pairs can be found in \cite{HOW1} (one can check also \cite[Lecture~5]{LI} and \cite{SCH}).

\begin{prop}[\cite{HOW1}]

Let $(\G, \G')$ be an irreducible reductive dual pair in $\O(\E, \b)$.
\begin{enumerate}
\item If the dual pair is of type I, there exists a division algebra $\mathbb{D}$ over $\mathbb{K}$ with involution $\iota$, vector spaces $\E_{1}$ and $\E_{2}$ over $\mathbb{D}$, sesquilinear forms $\b_{1}$ and $\b_{2}$ on $\E_{1}$ and $\E_{2}$ respectively which are both hermitian or skew-hermitian, such that $\E = \E_{1} \otimes_{\mathbb{D}} \E_{2}$ and $\G \approx \G(\E_{1}, \b_{1})$, $\G' \approx \G(\E_{2}, \b_{2})$.
\item Otherwise, we have $\E = \E_{1} \oplus \E^{*}_{1}$, and there exists a division algebra $\mathbb{D}$ over $\mathbb{K}$ and two vector spaces $\F_{1}$ and $\F_{2}$ over $\mathbb{D}$ such that $\E_{1} = \F_{1} \otimes_{\mathbb{D}} \F_{2}$ and $\G \approx \GL(\F_{1}, \mathbb{D})$, $\G' \approx \GL(\F_{2}, \mathbb{D})$.
\end{enumerate}
\label{PropositionDP}

\end{prop}
\begin{proof}

See \cite[Proposition~6.2~and~ 6.3]{HOW1}.

\end{proof}
As a corollary, we get the following dual pairs.

\begin{coro}
\label{ClassificationDP}
If $\mathbb{K} = \mathbb{C}$, we have:
\begin{enumerate}
\item $(\O(n, \mathbb{C}), \O(m, \mathbb{C}) \subseteq \O(nm, \mathbb{C})$, 
\item $(\Sp(2n, \mathbb{C}), \Sp(2m, \mathbb{C}) \subseteq \O(4nm, \mathbb{C})$,
\item $(\GL(n, \mathbb{C}), \GL(m, \mathbb{C}) \subseteq \O(2nm, \mathbb{C})$
\end{enumerate}
If $\mathbb{K} = \mathbb{R}$, we obtain:
\begin{enumerate}
\item $(\O(p_{1}, q_{1}, \mathbb{R}), \O(p_{2}, q_{2}, \mathbb{R})) \subseteq \O(p, q, \mathbb{R}), p = p_{1}p_{2}+q_{1}q_{2}, q = p_{1}q_{2} +p_{2}q_{1}$,
\item $(\U(p_{1}, q_{1}, \mathbb{C}), \U(p_{2}, q_{2}, \mathbb{C})) \subseteq \O(p, q, \mathbb{R}), p = 2(p_{1}p_{2}+q_{1}q_{2}), q = 2(p_{1}q_{2} +p_{2}q_{1})$,
\item $(\Sp(n_{1}, \mathbb{R}), \Sp(n_{2}, \mathbb{R})) \subseteq \O(p, p, \mathbb{R})$, $p  = 2n_{1}n_{2}$,
\item $(\O(n_{1}, \mathbb{C}), \O(n_{2}, \mathbb{C})) \subseteq \O(p, p, \mathbb{R})$, $p  = n_{1}n_{2}$, $n_{1}, n_{2} \neq 1$,
\item $(\Sp(n_{1}, \mathbb{C}), \Sp(n_{2}, \mathbb{C})) \subseteq \O(p, p, \mathbb{R})$, $p  = 4n_{1}n_{2}$,
\item $(\Sp(p_{1}, q_{1}, \mathbb{H}), \Sp(p_{2}, q_{2}, \mathbb{H})) \subseteq \O(p, q, \mathbb{R}), p = 4(p_{1}p_{2} + q_{1}q_{2}), q = 4(p_{1}q_{2} + p_{2}q_{1})$,
\item $(\O^{*}(n_{1}, \mathbb{H}), \O^{*}(n_{2}, \mathbb{H})) \subseteq \O(p, p, \mathbb{R}),  p = 2n_{1}n_{2}, n_{1}, n_{2} \neq 1$,
\item $(\GL(n_{1}, \mathbb{R}), \GL(n_{2}, \mathbb{R})) \subseteq \O(n_{1}n_{2}, n_{1}n_{2}, \mathbb{R})$,
\item $(\GL(n_{1}, \mathbb{C}), \GL(n_{2}, \mathbb{C})) \subseteq \O(2n_{1}n_{2}, 2n_{1}n_{2}, \mathbb{R})$,
\item $(\GL(n_{1}, \mathbb{H}), \GL(n_{2}, \mathbb{H})) \subseteq \O(4n_{1}n_{2}, 4n_{1}n_{2}, \mathbb{R})$.
\end{enumerate}

\end{coro}

In the rest of the paper, we consider dual pairs in the Pin-group: it's a non-trivial double cover of the orthogonal group $\O(\E, \b)$. This group is a subset of the invertible elements in the corresponding Clifford algebra of $(\E, \b)$. We recall now its construction.

We denote by $\T(\E)$ the tensor algebra of $\E$ by $\mathscr{T}$ the ideal generated by the elements of the form $x \otimes y + y \otimes x -\b(x, y).1, x, y \in \E$. The Clifford algebra, denoted by $\Cliff(\E, \b)$, is defined by
\begin{equation*}
\Cliff(\E, \b) = \T(\E) / \mathscr{T}.
\end{equation*}
The Clifford algebra inherits a natural structure of $\mathbb{Z}_{2}$-graded algebra coming from the gradation on $\T(\E)$. 

The natural composition $\E \to \T(\E) \to \Cliff(\E, \b)$ is an injection and we will identify the element $e \in \E$ with its image in $\Cliff(\E, \b)$. If we fix an orthonormal basis $\{e_{1}, \ldots, e_{n}\}$ of $\E$, the elements $e_{i_{1}}e_{i_{2}} \ldots e_{i_{k}}, 1 \leq i_{1} < \ldots < i_{k} \leq n$ form a basis of the algebra $\Cliff(\E, \b)$. For more details, one can check \cite[Chapter~5]{VAR} or \cite[Chapter~6.1]{GOD}.

On $\Cliff(\E, \b)$, we define the following involution $\alpha$:
\begin{equation*}
\alpha(e_{1} \ldots e_{k}) = (-1)^{k} e_{1} \ldots e_{k}
\end{equation*}
and the antiautomorphism $\tau$
\begin{equation*}
\tau(e_{1} \ldots e_{k}) = e_{k} \ldots e_{1}
\end{equation*}
The $\Pin$-group $\Pin(\E, \b)$ is defined by:
\begin{equation}
\Pin(\E, \b) = \left\{x \in \Cliff(\E, \b): x.\tau(\alpha(x)) = 1 \text{ and } \alpha(x)\gamma(\E)\tau(\alpha(x)) = \gamma(\E)\right\}.
\label{DefinitionPinGroup}
\end{equation}

The map $\pi: \Pin(\E, \b) \to \O(\E, \b)$ given by:
\begin{equation*}
\pi(x)(e) = \alpha(x)e\tau(\alpha(x)) \qquad (x \in \Pin(\E, \b), e \in \E),
\end{equation*}
is well-defined and surjective with kernel $\Ker(\pi) = \{\pm 1\}$. 

\begin{rema}

Let $(\E, \b)$ be a complex orthogonal space, and let $c$ be a $\mathbb{C}$-anti-linear involution such that the restriction of $\b$ to $\Fix(\E, c)$ is real valued. We denote by $\E_{0} = \Fix(\E, c)$ and $\b_{0}$ the restriction of $\b$ to $\E_{0}$ and $\tilde{c}$ the extension of $c$ to $\Cliff(\E, \b)$. In particular, we have $\Cliff(\E_{0}, \b_{0}) = \Fix(\Cliff(\E, \b), \tilde{c})$ and for all $e \in \E$, we have $\tilde{c}(e) = c(e)$.

Obviously, we have $\tilde{c}(\Pin(\E, \b)) = \Pin(\E, \b)$ and $\Pin(\E_{0}, \b_{0}) = \Fix(\Pin(\E, \b), \tilde{c})$.

\end{rema}

We recall the construction of the spinorial representation. From now on, we assume that $\E$ is an even dimensional vector space over $\mathbb{K}$.

\begin{defn}

Let $\S$ be a complex vector space and let $\gamma: \E \to \End(\S)$ be a linear map. Then, $(\S, \gamma)$ is a space of spinors for $(\E, \b)$ if 
\begin{enumerate}
\item $\gamma(x)\gamma(y) + \gamma(y)\gamma(x) = \b(x, y)$ for all $x, y \in \E$,
\item The only subspaces of $\S$ that are invariant under $\gamma(\E)$ are $0$ and $\S$.
\end{enumerate}

\label{DefinitionGamma}
\end{defn}

According to \cite[Section~6.1]{GOD}, we have by extension an application $\tilde{\gamma}: \Cliff(\E, \b) \to \End(\S)$ which is an isomorphism of algebra. In particular, the corresponding representation on $\S$ is irreducible. Moreover, up to conjugation, this space is unique. By inclusion of $\Pin(\E, \b) \subseteq \Cliff(\E, \b)$, we have an irreducible representation of $\Pin(\E, \b)$ in the space $\S$ and we denote by $(\Pi, \S)$ this representation.

\begin{defn}

If $\E$ is real, we denote by $\E_{\mathbb{C}}$ its complexification and $\b_{\mathbb{C}}$ the induced non-degenerate, symmetric, bilinear form on $\E_{\mathbb{C}}$. We have a natural inclusion $\Pin(\E, \b) \subseteq \Pin(\E_{\mathbb{C}}, \b_{\mathbb{C}})$ and let $(\Pi, \S)$ be the corresponding spinorial representation of $\Pin(\E_{\mathbb{C}}, \b_{\mathbb{C}})$. Then, the restriction of $\Pi$ to $\Pin(\E, \b)$ is still irreducible and that's how we define the spinorial representation of $\Pin(\E, \b)$.

\end{defn}

We finish this section by introducing some notations. If $\E$ is a real, we denote by $\SO(\E, \b)$ the special orthogonal group, i.e. the subgroups of $\O(\E, \b)$ consisting of elements of determinant $1$. Let $(p, q)$ be the signature of $\b$ and $\SO_{0}(\E, \b)$ be the connected component of $\SO(\E, \b)$ containing $\Id_{\E}$. We denote by $\Spin(\E, \b)$ the preimage of $\SO_{0}(\E, \b)$ in $\Pin(\E, b)$. 

If $(\E, \b)$ is complex, then, $\SO(\E, \b)$ is connected and we denote by $\Spin(\E, \b)$ its preimage in $\Pin(\E, \b)$.

In the next section, we study the dual pairs in the $\Pin$-group. More particularly, starting with an irreducible dual pair $(\G, \G')$ in $\O(\E, \b)$, we determine when the pull-back of $(\G, \G')$ is a dual pair in $\Pin(\E, \b)$.

\section{Dual pairs in the $\Pin$-group}

\label{SectionDualPairs}

In this section, we first prove that all pull-backs of dual pairs in $\O(\E, \b)$ are dual pairs in $\Pin(\E, \b)$. Then, we identify the isomorphism class of such pairs. Throughout this section if a real orthogonal group  $\O(\E, \b)$ is given then $\pi$ denotes the natural projection $\Pin(\E, \b)\to \O(\E, \b)$.

\subsection{Pull-back of dual pairs}

We begin our work with two topological lemmas about dual pairs.
 
\begin{lemme}
\label{comdual}

Let $\pi:\tilde{\S}\to \S$ be a discrete covering of connected Lie groups and $(\G,\G')$ be a dual pair in $\S$. Then $(\pi^{-1}(\G),\pi^{-1}(\G'))$ is a dual pair in $\tilde{\S}$ if and only if $\pi^{-1}(\G)$ and $\pi^{-1}(\G')$ commute. 

\end{lemme}

\begin{proof}

Assume that $\pi^{-1}(\G)$ and $\pi^{-1}(\G')$ commute. Let $\tilde{s}\in \tilde{\S}$ be an element commuting with $\pi^{-1}(\G')$ then $\pi(\tilde{s})$ commutes with $\G'$ and therefore belongs to $\G$. As a result, $Z_{\tilde{\S}}(\pi^{-1}(\G'))$ is contained in $\pi^{-1}(\G)$. Since the other inclusion is true by assumption,  we have equality. By symmetry we also have $Z_{\tilde{\S}}(\pi^{-1}(\G))=\pi^{-1}(\G')$ and we are done.

\end{proof}

\begin{lemme}

Let $\pi:\tilde{\S}\to \S$ be a discrete covering of connected Lie groups and $(\G,\G')$ be a dual pair in $\S$. If $\G$ and $\G'$ are connected then $(\pi^{-1}(\G),\pi^{-1}(\G'))$ is a dual pair in $\tilde{\S}$.

\label{LemmaConnected}

\end{lemme}

\begin{proof}

Because of the preceding lemma, it suffices to show that any element of $\pi^{-1}(\G)$ commutes with any element of $\pi^{-1}(\G')$. Chose an arbitrary set-theoretic section $\S\to \tilde{\S}$ sending an element $s$ to one of its lift $\tilde{s}$ in $\tilde{\S}$. We define a function :
\begin{equation*}
\psi:\left|\begin{array}{rcl}\G\times \G'& \to&Z(\tilde{\S})\\ (g,g')&\mapsto& \tilde{g}\tilde{g'}\tilde{g}^{-1}\tilde{g'}^{-1}\end{array}\right.
\end{equation*}
Since $\tilde{\S}\to \S$ is a local homeomorphism and $\psi$ does not depend on the chosen lift, $\psi$ is continuous. The group $\G\times \G'$ being connected and $Z(\tilde{\S})$ being discrete, it follows that $\psi$ is constant. Since $\psi([1_{\G},1_{\G'}])=1_{\tilde{\S}}$, it follows that $\psi=1_{\tilde{\S}}$ which implies that $\pi^{-1}(\G)$ and $\pi^{-1}(\G')$ commute. Whence the result.

\end{proof}

Now we can prove Theorem \ref{prop:dualift} that we recall here 

\dualift*

\begin{proof}

Because of the preceding lemma about connectedness, the lift of $(\G, \G')$ in  $\Pin(\E, \b)$ is a dual pair when $(\G, \G')$ is (2), (3), (5), (6) or (7) in Type I or (9), (10) in Type II (see Section \ref{SectionPreliminaries}, Corollary \ref{ClassificationDP}). Thus we are left with two kinds, the complex orthogonal dual pairs (type I) and the real linear ones (type II). In this proof $\pi$ denotes the projection of $\Pin(\E, \b)$ onto $\O(\E, \b)$. 

\begin{itemize}
\item The $(\O(n,\mathbb{C}), \O(m,\mathbb{C}))$-case ($n,m\geq 2$). 

\noindent Let $e_1,\dots,e_n$ and $f_1,\dots, f_m$ be orthogonal bases for $\mathbb{C}^n$ and $\mathbb{C}^m$ respectively. Then we denote 
\[k_{(m-t)n+s}:=e_s\otimes f_t\text{ and } \ell_{(m-t)n+s}:=\sqrt{-1}e_s\otimes f_t\]
for $1\leq s\leq n$ and $1\leq t\leq m$. It is an easy exercise to check that $(k_1,\dots,k_{nm},\ell_1,\dots,\ell_{nm})$ is an orthogonal basis for $\R^{2nm}$ such that the $k$'s are of norm $1$ and the $\ell$'s are of norm $-1$. 

\noindent Let $g=k_1\dots k_{(m-1)n+1}\ell_1\dots \ell_{(m-1)n+1}$ in the Clifford algebra of $\R^{2nm}$. Then $\pi(1)$ and $\pi(g)$ are elements  of the two different connected components in $\O(n,\mathbb{C})$. 

\noindent Likewise, if we define $h=k_1\dots k_n \ell_1\dots \ell_n$  then $\pi(1)$ and $\pi(h)$ are elements  of the two different connected components in $\O(m,\mathbb{C})$. 

\noindent We want to prove that $\pi^{-1}(\O(n,\mathbb{C}))$ and $\pi^{-1}(\O(m,\mathbb{C}))$ commute. Thus we want to prove that the following application is constantly trivial: \[\psi:\left|\begin{array}{rcl} \O(n,\mathbb{C}) \times \O(m,\mathbb{C})&\to& Z(\Pin(nm,nm,\R))\\ (u,v)&\mapsto& [\tilde{u},\tilde{v}]\end{array}\right.\]
Since $\psi$ is continuous and $Z(\Pin(nm,nm,\R))$ is discrete, the application factors through $\pi_0(\O(n,\mathbb{C}))\times \pi_0(\O(m,\mathbb{C}))$. As a result, it suffices to compute $\psi(x,y)$ for $(x,y)\in\{1,g\}\times \{1,h\}$ and, at the end, just $\psi(g,h)$. Explicitly, 
\begin{align*}
gh&=k_1\dots k_{(m-1)n+1}\ell_1\dots \ell_{(m-1)n+1}k_1\dots k_n \ell_1\dots \ell_n\\
&=(-1)^{(n-1)(2m-1)}k_1\dots k_n k_{n+1}\dots k_{(m-1)n+1}\ell_1\dots \ell_{(m-1)n+1}k_1 \ell_1\dots \ell_n\\
&=(-1)^{(n-1)(2m-1)}(-1)^{(m-1)(m+1)}k_1\dots k_n\ell_1\dots \ell_{(m-1)n+1}k_1 k_{n+1}\dots k_{(m-1)n+1} \ell_1\dots \ell_n\\
&=(-1)^{2(n-1)(2m-1)}(-1)^{(m-1)(m+1)}k_1\dots k_n\ell_1\dots \ell_n\ell_{n+1}\dots \ell_{(m-1)n+1}k_1\dots k_{(m-1)n+1} \ell_1\\
&=(-1)^{2(n-1)(2m-1)}(-1)^{2(m-1)(m+1)}k_1\dots k_n\ell_1\dots \ell_nk_1\dots k_{(m-1)n+1} \ell_1\ell_{n+1}\dots \ell_{(m-1)n+1}\\
&=hg.
\end{align*}
As a result $\psi(g,h)=1$, i.e. the lifts commute. 

\noindent Because of Lemma \ref{comdual}, $(\pi^{-1}(\O(n,\mathbb{C})),\pi^{-1}(\O(m,\mathbb{C})))$ is a dual pair in $\Pin(n,n,\R)$

\item The $(\GL(n,\mathbb{R}), \GL(m,\mathbb{R}))$-case. 

\noindent Let $s\in \GL(n,\R)$ be $\left(\begin{array}{c|c} -1&0\\\hline 0&I_{n-1}\end{array}\right)$ and $t\in \GL(m,\R)$ be $\left(\begin{array}{c|c} -1&0\\\hline 0&I_{m-1}\end{array}\right)$. Then $s$ (resp. $t$) is contained in the non-trivial component of $\GL(n,\R)$ (resp. $\GL(m,\R)$). In order to do this one needs to rewrite these elements in an orthogonal basis of non-isotropic vectors that is $b_{i,j,\pm}=\frac{1}{\sqrt{2}}\left(e_i\otimes f_j\pm e_i^*\otimes f_j^*\right)$. 

\noindent Then a pull-back of the aforementioned elements in $\Pin(nm,nm,\R)$  is given by 
\begin{align*}
\tilde{s}=b_{1,1,+}b_{1,1,-}\cdots b_{1,m,+}b_{1,m,-}\text{ and }\tilde{t}=b_{1,1,+}b_{1,1,-}\cdots b_{n,1,+}b_{n,1,-}.
\end{align*}
Making the same kind of computation as in the previous case, one has $\tilde{s}\ \ \tilde{t}=\tilde{t}\ \ \tilde{s}$ and using the same reasoning as in the previous case we know that the lift of  $(\GL(n,\mathbb{R}), \GL(m,\mathbb{R}))$ is a dual pair in $\Pin(nm,nm,\R)$. 
\end{itemize}

\end{proof}

\subsection{Identifying the isomorphism class of pull-backs}

Before stating the theorem, we recall some well-known facts about lifts of unitary groups in general. If $n\geq 1$, then $\U(n)$ has only two extensions by a group of order $2$, the trivial one (i.e. direct product) and another one given by the following formula :
\[\{(g, \xi) \in \U(n) \times \mathbb{C}^{*}, \xi^{2} = \det(g)\}.\]

If $p,q\geq 1$, there are four different lifts for $\U(p,q)$ and only one of them is non-trivial above both $\U(p)$ and $\U(q)$ (the maximal compact subgroup of $\U(p,q)$ is $\U(p)\times \U(q)$ and $\H^{2}(\U(n), \{\pm 1\})$ is of cardinal 2). We  denote this unique extension by $\Lambda(p,q)$. 

Our goal in this section is to prove Proposition \ref{theo:classlift} :

\classlift*

The proof {\it per se} of the theorem is in Section \ref{prooftheo}. Section \ref{prelrem} is devoted to a few technical lemmas needed in the proof. In Section \ref{inducprop}, we prove Proposition \ref{Decomp} which is largely inspired by Theorem 2.9 in \cite{SLU}. This proposition allows us to decompose the extensions arising from the pull-back of dual pairs into smaller pieces. 

\subsubsection{Preliminary remarks}\label{prelrem}

\begin{lemme}

\label{redcomp}

Let $1 \to \{\pm 1\} \to \G' \to \G \to 1$ be a discrete extension of a reductive Lie group $\G$ and $\K$ be a maximal compact subgroup of $\G$. Then the extension $\G'$ is trivial if and only if the induced extension of $\K$ is trivial. 

\end{lemme}

\begin{proof}

First remark that the extension is trivial if and only if the topological covering $\G' \to \G$ is trivial, i.e. $\G \times \{\pm 1\} \to \G$. Because the fact that the extension is trivial boils down to a topological feature and that there is a topological retraction from any reductive Lie group to any of its maximal compact subgroup, the result follows.

\end{proof}

From the lemma above, it follows :

\begin{coro}

\label{coroconsimcon}

Let $\G$ be a connected, simply connected subgroup of $\O(\E, \b)$ then its pull-back in $\Pin(\E, \b)$ is a trivial extension. 

\end{coro}

\begin{proof}

Because $\G$ is connected and simply connected any discrete lift of $\G$ has to be disconnected whence leads to a trivial extension by Lemma \ref{redcomp}. 

\end{proof}

\begin{lemme}

\label{trivialpullback}

Let $\pi: \G' \to \G$ be surjective morphism with kernel $\C=\langle c\rangle$ of order $2$. Let $\Delta=\langle (c,c)\rangle$ in $\G'\times \G'$. Let $\H$ be a subgroup of $\G$ and $\H_{0}$ be the diagonal embedding of $\H$ in $\G \times \G$. 

Then the pull-back of $\H_{0}$ in $\G' \times \G'/\Delta$  is a trivial extension of $\H$ by $\C$. 

\end{lemme}

\begin{proof}

Let $\P_{0}$ be the pull-back of $\H_{0}$  in $\G' \times \G'/\Delta$ and $\P$ be the pull-back of $\H$ in $\G'$. Then $\P_{0}$ contains the following subgroup :
\[\F_{1}:=\{(p,p)\mid p\in \P\}/\Delta.\]
It is clear that $\F_{1}$ is isomorphic to $\H$. $\P_{0}$ also contains $\F_{2}:=\langle (1,c)\mod \Delta\rangle$. Verifying that $\F_{1}$ and $\F_{2}$ are direct factors of $\P_{0}$ is easily done and the result follows. 

\end{proof}

\subsubsection{The induction proposition}\label{inducprop}

Because the dual pairs we are interested in always lie $\SO(\E, \b)$, the framework is simpler than in \cite[Section~2]{SLU}. 

Let $\iota : \G \to \SO(\E, \b)$ be an inclusion then we define $\Ex_{\iota}(\G)$ to be the extension given by the pull-back of $\G$ in $\Pin(\E, \b)$ :
\[1 \to \{\pm 1\} \to \pi^{-1}(\G) \to \G \to 1.\]
Classically, we can associate to any class of such central extension an element (still denoted $\Ex_{\iota}(\G)$) in $\H^2(\G, \{\pm 1\})$. 

\begin{rema}

To any such extension we can associate a $2$-cocycle $z_{\iota}(\G)$ (whose class is denoted by $\Ex_{\iota}(\G)$) by choosing a set-theoretic section $s: \G \to \pi^{-1}(\G)$ and defining :
\[z_{\iota}(\G)(g,h):=s(g)s(h)s(gh)^{-1}.\]
The fact that $ z_{\iota}(\G)$ is a $2$-cocycle for $\G$ with values in $\{\pm 1\}$ is a classical abstract non-sense. While non-trivial, proving that this induces a bijection between classes of extensions and $\H^{2}(\G, \{\pm 1\})$ is a classical matter that we leave to the reader (see for instance \cite{BRO}). 

\end{rema}

In our case, dual pairs in $\O(\E, \b)$  are always indexed by two integers or two pairs of integers $(n,m)$, so that we can write them all as $(\G, \G')=(\G_n, \G_m)$. For a given $\G_n$, $\iota_m: \G_n \to \O(\E, \b)$ is the inclusion for which $(\G_n, \G_m)$ is a dual pair in $\O(\E, \b)$. One easily sees from the construction of such dual pairs (this construction is recalled in proposition \ref{PropositionDP}) that if $m_1+m_2=m$ and we denote $\iota_{m_i}: \G_n \to \O(\E_{i}, \b_{i})$, then we have an orthogonal sum:
\[\E = \E_1\oplus \E_2.\]
The wanted decomposition is given in the following proposition :

\begin{prop}

\label{Decomp}

Assume $n, m_1, m_2$ are non-zero integers or pair of integers and $\G_*$ be any series of groups arising in the list of dual pairs  in $\O(\E, \b)$ except the real orthogonal one in Corollary \pageref{ClassificationDP}. Then 
\[\Ex_{\iota_{m_1+m_2}}(\G_n) = \Ex_{\iota_{m_1}}(\G)\Ex_{\iota_{m_2}}(\G).\]

\end{prop}

\begin{proof}

We use the formalism of $2$-cocycles to prove this. We recall, that if $\E = \E_1 \oplus \E_2$ (orthogonal sum of quadratic spaces) then we have an identity on the Clifford algebra (see \cite{ABS}):
\begin{equation*}
\Cliff(\E, \b) = \Cliff(\E_{1}, \b_{1}) \hat{\otimes} \Cliff(\E_{2}, \b_{2}).
\end{equation*}
The product is defined on homogenous  components relative to the natural $\mathbb{Z}_{2}$-grading, if $c_1,d_1\in \Cliff(\E_{1}, \b_{1})$ are of grading $u_1,v_1$ and $c_{2}, d_{2} \in \Cliff(\E_{2}, \b_{2})$ are of grading $u_2, v_2$ then 
\begin{equation*}
(c_{1} \otimes c_{2}) \times (d_{1} \otimes d_{2}) = (-1)^{v_{1}u_{2}}(c_{1}d_{1}) \otimes (c_{2}d_{2}).
\end{equation*}
For an element $c \in \Pin(\E, \b)$, the $\mathbb{Z}_{2}$ grading is given by 0 if $\det_{\E}(c) = 1$ and $1$ otherwise.

In particular, if $c_i, d_i$ are above $\SO(\E_{i}, \b_{i})$, we end up with the formula 
\begin{equation}
\label{equationdecorth}
(c_1\otimes c_2) \times (d_1\otimes d_2) = (c_1d_1)\otimes (c_2d_2).
\end{equation}

Now the proposition is just some abstract non-sense. We denote $s_i$ a section for $\pi_i:\pi_i^{-1}(\iota_{m_i}(\G_n)) \to \G_n$. Let $s$ be the following section of   $\pi: \pi^{-1}((\iota_{m_1+m_2})(\G_n)) \to \G_n$ defined by $s(g):= s_1(g)\otimes s_2(g)$. If we are given two elements $s_1(g)\otimes s_2(g)$ and $s_1(h)\otimes s_2(h)$, then we have 
\begin{align*}
s(g)s(h)&=(s_1(g)\otimes s_2(g))(s_1(h)\otimes s_2(h))\\
&=s_1(g)s_1(h)\otimes s_2(g)s_2(h)\text{ (because  of Equation \ref{equationdecorth})}\\
&=z_{\iota_{m_1}}(\G)s_1(gh)z_{\iota_{m_2}}(\G)s_2(gh)\\
&=z_{\iota_{m_1}}(\G)z_{\iota_{m_2}}(\G)\underbrace{s_1(gh)\otimes s_2(gh)}_{s(gh)}\end{align*}
So that, finally,  in $\H^2(\G, \{\pm 1\})$, 
\begin{equation*}
\Ex_{\iota_{m_1+m_2}}(\G) = \Ex_{\iota_{m_1}}(\G)Ex_{\iota_{m_2}}(\G).
\end{equation*}

\end{proof}

In the real orthogonal case, it is slightly more complicated because one has to take the determinant into account (whence the use of graded extensions in \cite{SLU}). In our case, it should be remarked that $\H^{2}(\G, \{\pm1\})$ is of exponent $2$. 

\subsubsection{Proof of the classification theorem}\label{prooftheo}

We first deal with a few easy cases of type $I$. 

\begin{prop}

Let $(\G, \G')$ be one of the following dual pairs in $\O(\E, \b)$ :
\begin{itemize}
\item $(\Sp(2n, \mathbb{C}),\Sp(2m,\mathbb{C}))$ ($n,m\geq 1$),
\item $(\Sp(p_1, q_1, \mathbb{H}), \Sp(p_2, q_2, \mathbb{H}))$, ($p_1,q_1,p_2,q_2\geq 0$),
\item $(\O^{*}(n, \mathbb{H}), \O^{*}(m, \mathbb{H}))$,  ($n,m\geq 1$).
\end{itemize}
Then, the lift of $(\G, \G')$ in $\Pin(\E, \b)$ is trivial. 

\end{prop}

\begin{proof}

In any of these cases (see \cite{VIN} or \cite{KNA} for instance), the groups involved are connected and simply connected. It follows from Corollary \ref{coroconsimcon} that the extensions need to be trivial. 

\end{proof}

The next cases involve more computations.

\begin{prop}

If $(\G, \G')$ is one of the following dual pairs in $\O(\E, \b)$ :
\begin{itemize}
\item $(\O(n, \mathbb{C}), \O(m, \mathbb{C}))$ ($n,m\geq 2$) or
\item $(\Sp(2n,\mathbb{R}), \Sp(2m,\mathbb{R}))$ ($n,m\geq 1$),
\end{itemize}
then the lift of $(\G, \G')$ in $\Pin(\E, \b)$ is trivial. 

\end{prop}

\begin{proof}

Because of Proposition \ref{Decomp}, we only need to prove that the lift of $\O(n, \mathbb{C})$ in $\O(n,n,\R)$ (although, it is not a dual pair if $m=1$, we  still use this inclusion as a building block to understand extensions for $m \geq 2$)  and $\Sp(2n,\mathbb{R})$ in $\O(2n,2n,\R)$ are trivial. 

\begin{itemize}
\item Let $(e_1,\dots,e_n)$ be an orthonormal basis of $(\mathbb{C}^n,b)$. Then $(e_1,\dots,e_n,\sqrt{-1}e_1,\dots,\sqrt{-1}e_n)$ is a real basis of $(\mathbb{C}^n)_{\R}$. Furthermore this is an orthogonal basis for the bilinear form $\Re(b)$ (the first $n$ vectors are of norm $1$ and the remaining ones are of norm $-1$). From this realification of the quadratic space follows the inclusion of $\O(n, \mathbb{C})$ in $\O(n, n, \R)$.

By the very definition of the inclusion of $\O(n, \R)$ in $\O(n, \mathbb{C})$, any element of $\O(n, \R)$ commutes with the multiplication by $\sqrt{-1}$. As a result $\O(n, \R)$ is diagonally contained in $\O(n, \R) \times \O(n, \R)$ the maximal compact subgroup of $\O(n, n, \R)$ and therefore (via Lemma \ref{trivialpullback}) has a trivial pull-back in $\Pin(n, n, \R)$. Since $\O(n, \R)$ is the maximal compact subgroup of $\O(n, \mathbb{C})$, Lemma \ref{redcomp} implies that the pull-back of $\O(n, \mathbb{C})$ in $\Pin(n, n, \R)$ needs to be trivial as well. 

\item In the symplectic case, Lemma \ref{redcomp} implies that we only need to focus on the pull-back of $\U(n)$ (a maximal compact subgroup of $\Sp(2n,\R)$) in $\Pin(2n, 2n, \R)$ via the inclusion of $\Sp(2n, \R)$ in $\O(2n, 2n, \R)$. Then, similarly to the preceding case, we diagonally include $\U(n)$ in $\O(2n, \R)\times \O(2n, \R)$ and using Lemma \ref{trivialpullback}, we have a trivial extension in $\Pin(2n, 2n, \R)$ as well. 
\end{itemize}

\end{proof}

It should be remarked that, in the symplectic case, we have only two choices for the pull-back of $\Sp(2n,\R)$, either the trivial one or the metaplectic one. Since the metaplectic group is not linear while $\Pin(2n, 2n, \R)$ is, we could have directly conclude that the pull-back of $\Sp(2n,\R)$ had to be the trivial one. 

We end the study of the cases of type $I$ with the unitary dual pairs. 

\begin{prop}

Let $p_1,q_1,p_2,q_2$ be non-negative integers so that $p_1q_1\neq 0$ and $p_2q_2\neq 0$. The lift of $\U(p_1,q_1)$  in $\Pin(2(p_1p_2+q_1q_2),2(p_1q_2+q_1p_2),\R)$ arising from the dual pair $(\U(p_1,q_1),U(p_2,q_2))$ in $\O(2(p_1p_2+q_1q_2),2(p_1q_2+q_1p_2),\R)$  is isomorphic to 
\[\left\lbrace\begin{array}{ll}\Lambda(p_1,q_1)&\text{if $p_2+q_2$ is odd}\\ \U(p_1,q_1)\times \{\pm 1\}&\text{ otherwise}\end{array}\right.\]

\end{prop}

\begin{proof}

Because of Proposition \ref{Decomp} and the fact that $\H^{2}(\U(p,q), \{\pm 1\})$ is of exponent $2$,  the result  directly follows if we prove that for the cases $(p_2,q_2)=(1,0)$ and $(0,1)$ the lift of $\U(p_1,q_1)$ in the corresponding pin-group is $\Lambda(p_1,q_1)$. 

First remark that these two cases respectively correspond to the inclusion of $\U(p_1,q_1)$ in $\O(2p_1,2q_1,\R)$ and $\O(2q_1,2p_1,\R)$. Since these last two groups are isomorphic, it suffices to prove it for the inclusion of $\U(p_1,q_1)$ in $\O(2p_1,2q_1,\R)$. 

Because of Lemma \ref{redcomp}, we focus on the inclusion of the maximal compact subgroup $\U(p_1)\times \U(q_1)$ in $\O(2p_1,\R)\times \O(2q_1,\R)\leq \O(2p_1,2q_1,\R)$. Since it is known that the lift of $\U(n)$ in $\Pin(2n,\R)$ via the inclusion of $\U(n)$ in $\O(2n,\R)$ (see e.g. \cite{MEI}, page 77, remark 3.9), it follows that the lift $\U(p_1,q_1)$ in $\Pin(2p_1,2q_1,\R)$ is both non-trivial when restricted to $\U(p_1)$ and $\U(q_1)$, this lift is therefore $\Lambda(p_1,q_1)$ by definition.

\end{proof}

Then we deal with the case of type $II$. 

\begin{prop}

Let $(n,m)$ be a pair of positive integers and $\mathbb{K} = \R$, $\mathbb{C}$ or $\mathbb{H}$. Then the lift of the dual pair $(\GL(n,\mathbb{K}),\GL(m,\mathbb{K}))$ is trivial. 

\end{prop}

\begin{proof}

Because of Proposition \ref{Decomp}, the result  will follow if we prove that the lift of $\GL(n,\mathbb{K})$ in $\Pin(n\dim_{\R}\mathbb{K},n\dim_{\R}\mathbb{K},\R)$ is trivial. 

The inclusion $f: \GL(n,\mathbb{K})\to \O(n\dim_{\R}\mathbb{K}, n\dim_{\R}\mathbb{K},\mathbb{R})$ is given by putting the usual quadratic form of signature $(n\dim_{\R}\mathbb{K}, n\dim_{\R}\mathbb{K})$ in $\mathbb{K}^n\oplus (\mathbb{K}^n)^*$. More precisely, if $\iota_{\K}:\GL(n,\mathbb{K})\to \GL(n\dim_{\R}\mathbb{K},\R)$ denotes the reelification of the matrix then the inclusion is defined by \[f(A)=\left(\begin{array}{c|c} \iota_{\K}(A)&   0 \\ \hline 0& \text{ } ^t\iota_{\K}(A)^{-1}\end{array} \right)\text{ for $A\in \GL(n,\mathbb{K})$.}\]

Let $\sigma_{\mathbb{K}}$ be the automorphism of $\GL(n,\mathbb{K})$ induced by the conjugation in $\mathbb{K}$ (i.e. the identity if $\mathbb{K}=\mathbb{R}$, the complex conjugation if $\mathbb{K}=\mathbb{C}$ or the quaternionic conjugation). One may directly check that 
\begin{equation}
\label{transpconj}
\iota_{\mathbb{K}}(\sigma_{\mathbb{K}}( \ ^tA))=\ ^t\iota_{\mathbb{K}}(A).
\end{equation}

Let $\K_{\mathbb{K}}$ be \[\{g\in \GL(n,\mathbb{K})\mid g\ ^t\sigma_{\mathbb{K}}(g)=I_n\}.\] It is maximal compact in $\GL(n,\mathbb{K})$. By Lemma \ref{redcomp}, we only need to show that the lift of $f(\K_{\mathbb{K}})$ is trivial. Let $A$ be in $\K_{\mathbb{K}}$, then
 \begin{align*}
^t\iota_{\K}(A)^{-1}&=^t\iota_{\K}(A^{-1})\\
&=\iota_{\mathbb{K}}(\sigma_{\mathbb{K}}( \ ^tA^{-1}))\text{ because of Equation \ref{transpconj}}\\
&=\iota_{\mathbb{K}}(A)\text{ because $A\in K_{\mathbb{K}}$.} \end{align*}

Whence, for  $A\in \K_{\mathbb{K}}$, we have 
\[f(A)=\left(\begin{array}{c|c} \iota_{\K}(A)& 0   \\ \hline 0& \text{ } \iota_{\K}(A)\end{array} \right).\]
Then, changing the basis of isotropic vectors $(\underline{e},\underline{e^*})$ to the basis $(\underline{e+e^*},\underline{e-e^*})$, it follows easily that $\K_{\mathbb{K}}$ is diagonally embedded via $f$ in some conjugate of $\O(n\dim_{\R}\mathbb{K},\ \mathbb{R})\times O(n\dim_{\mathbb{R}}\mathbb{K},\mathbb{R})$.  Applying Lemma \ref{trivialpullback}, the pull-back of $f(\K_{\mathbb{K}})$ in $\Pin(n\dim_{\R}\mathbb{K},n\dim_{\R}\mathbb{K},\mathbb{R})$ is trivial. Whence the result with Lemma \ref{redcomp}. 

\end{proof}

\begin{rema}

In Section \ref{SectionPreliminaries}, Corollary \ref{ClassificationDP}, we recalled the classification of dual pairs in the complex orthogonal group. The nature of the double covers can be obtained using the method we use in this section. 

Let $(\GL(n, \mathbb{C}), \GL(m, \mathbb{C}))$ be the dual pair in $\O(2nm, \mathbb{C})$. The nature of the preimage of $\GL(n, \mathbb{C})$ in $\Pin(2nm, \mathbb{C})$ can be determined by studying the nature of the preimage of $\GL(n, \mathbb{C})$ is $\Pin(2n, \mathbb{C})$. This double cover is uniquely determined by its restriction to $\U(n, \mathbb{C})$. The embedding of $\U(n, \mathbb{C})$ is $\O(2n, \mathbb{C})$ is included in $\SO(2n, \mathbb{C})$, and so the double cover is non-trivial and isomorphic to $\det^{\frac{1}{2}}$ (\cite[page~77, ~Remark~3.9]{MEI}). In particular, the double cover of $\GL(n, \mathbb{C})$ in $\Pin(2nm, \mathbb{C})$ is isomorphic to the $\det^{\frac{1}{2}}$-cover. Moreover, according to Lemma \ref{LemmaConnected}, the preimages of $\GL(n, \mathbb{C})$ and $\GL(m, \mathbb{C})$ commute in $\Pin(2nm, \mathbb{C})$ and then form a dual pair.

For the dual pair $(\G = \Sp(2n, \mathbb{C}), \G' = \Sp(2m, \mathbb{C}))$ in $\O(4nm, \mathbb{C})$, we have, according to Lemma \ref{LemmaConnected} and Corollary \ref{coroconsimcon}, that $(\tilde{\G}, \tilde{\G'})$ is a dual pair in $\Pin(4nm, \mathbb{C})$ and both covers are trivial.

Finally, the dual pair $(\G = \O(n, \mathbb{C}), \G' = \O(m, \mathbb{C}))$ in $\O(nm, \mathbb{C})$ is a consequence of \cite{SLU}. According to \cite[Theorem~3.4]{SLU}, if $n$ and $m$ are both odd, $(\tilde{\G}, \tilde{\G'})$ is a dual pair in $\Pin(nm, \mathbb{C})$ and if one of this two integers is even, then $(\tilde{\G}, \tilde{\G'_{0}})$ is a dual pair in $\Pin(nm, \mathbb{C})$, where $\tilde{\G'_{0}} = \pi^{-1}(\SO(m, \mathbb{C}))$. The nature of the covers of $\tilde{\G}$ and $\tilde{\G'}$ is obtained by complexification of the one obtained in \cite[Corollary~2.13]{SLU}.

\label{CoverComplexPairs}

\end{rema}

\begin{rema}

The double covers for the pair $(\widetilde{\GL(n, \mathbb{R})}, \widetilde{\GL(m, \mathbb{R})})$ in $\Pin(nm, nm, \mathbb{R})$ can be obtained via the pair $(\widetilde{\GL(n, \mathbb{C})}, \widetilde{\GL(m, \mathbb{C})})$ in $\O(2nm, \mathbb{C})$. The restriction of the $\det^{\frac{1}{2}}$-cover to $\GL(n, \mathbb{R})$ is trivial.

In this case, the nature of the double cover of $\GL(n, \mathbb{R})$ in $\O(n, n, \mathbb{R})$ was already known (one can check \cite[Equation~(2.3)]{HOW5}).

\end{rema}

\section{Duality for the spinorial representation}

\label{SectionDualityPin}

Let $(\E, \b)$ be an even dimensional vector space over $\mathbb{K} = \mathbb{R}$ or $\mathbb{C}$ and let $(\Pi, \S)$ be the corresponding spinorial representation of $\Pin(\E, \b)$ as constructed in Section 1. To simplify the notations, we denote by $\E_{\mathbb{C}}$ the complexification of $\E$ if $\E$ is real and $\E$ if $\E$ is already complex.

Let $\tilde{\Pi}: \Pin(\E, \b) \to \GL(\End(\S))$ be the representation we have by conjugation:
\begin{equation*}
\tilde{\Pi}(c)(X) = \Pi(c)X\Pi(c)^{-1}, \qquad (c \in \Pin(\E, \b), X \in \Cliff(\E, \b)),
\end{equation*}
and $\tilde{\rho}: \Pin(\E, \b) \to \GL(\Cliff(\E, \b))$ be the representation given by:
\begin{equation*}
\tilde{\rho}(c)(X) = cXc^{-1} \qquad (c \in \Pin(\E, \b), X \in \Cliff(\E, \b)).
\end{equation*}
Obviously, both $\Pin(\E, \b)$-modules $\tilde{\Pi}$ and $\tilde{\rho}$ are equivalent. We denote by $\tilde{\gamma}$ the isomorphism between $\Cliff(\E, \b)$ and $\End(\S)$ defined in Section \ref{SectionPreliminaries}. Then, for every subgroup $\G$ of $\O(\E, \b)$, we have:
\begin{equation*}
\tilde{\gamma}(\End(\S)^{\tilde{\G}}) = \Cliff(\E, \b)^{\tilde{\G}},
\end{equation*}

Let $\tilde{\rho_{_1}}$ be the natural action of $\Pin(\E, \b)$ on $\Lambda(\E)$ obtained by extension of the action $\pi$ of $\Pin(\E, \b)$ on $\E$ and consider the linear isomorphism (\cite[Section~5.3]{VAR}):
\begin{equation*}
\T: \Lambda(\E) \ni v_1\wedge \cdots \wedge v_k \to \frac{1}{k!}\sum_{\sigma\in \mathscr{S}_k} \epsilon(\sigma)v_{\sigma(1)}\cdots v_{\sigma(k)} \in \Cliff(\E, \b).
\end{equation*}

Let $c$ be in $\widetilde{\SO(\E, \b)}$ and $v_1\wedge \cdots\wedge v_k$ be in $\Lambda(\E)$. Then 

\begin{eqnarray*}
\T(\tilde{\rho_{1}}(c)(v_1\wedge \cdots\wedge v_k)) & = & \T(\pi(c)(v_1)\wedge\cdots\wedge \pi(c)(v_k)) = \sum_{\sigma \in \mathscr{S}_{k}} \epsilon(\sigma)\pi(c)(v_{\sigma(1)})\cdots \pi(c)(v_{\sigma(k)}) \\
                                                                                 & = & \sum_{\sigma \in \mathscr{S}_{k}} \epsilon(\sigma) cv_{\sigma(1)}c^{-1}\cdots cv_{\sigma(k)}c^{-1} = \sum_{\sigma \in \mathscr{S}_{k}}\epsilon(\sigma)cv_{\sigma(1)}\cdots v_{\sigma(k)}c^{-1}\\
                                                                                 & = & \tilde{\rho}(c)\T(v_1\wedge \cdots\wedge v_k)
\end{eqnarray*}
This means that $\T$ is an isomorphism of $\widetilde{\SO(\E, \b)}$-modules. It follows immediately that
\begin{equation}
\End(\S)^{\tilde{\G}} = \T \circ \tilde{\gamma}^{-1}(\Lambda(\E)^{\tilde{\G}}).
\label{TransferOfInvariants}
\end{equation}

\begin{rema}
\begin{enumerate}

\item For any irreducible representation $(\rho,V)$ of $\Pin(\E,\b)$ and any central element $c$ in $\Pin(\E,\b)$, $\rho(c)$ is scalar because of Schur's lemma. It follows that the conjugation action of $\Pin(\E,\b)$ on $\End(V)$ factors through an action of $\O(\E,\b)$. 
\item Similarly, let $\pr\tilde{\Pi}$, $\pr\tilde{\rho}$ and $\pr\tilde{\rho_{1}}$ the representation of $\G$ we obtain via $\tilde{\Pi}$, $\tilde{\rho}$ and $\tilde{\rho_{1}}$. Then, for all $g \in \SO(\E, \b)$, we have:
\begin{equation*}
\pr\tilde{\Pi}(g) = \pr\tilde{\rho}(g) = \pr\tilde{\rho_{1}}(g).
\end{equation*}
\end{enumerate}
\label{RemarkInvariants}
\end{rema}

We now explain precisely the method we use. Let $(\G, \G')$ be a reductive dual pair in $\O(\E, \b)$. According to Section 2, for the dual pairs we consider, the preimages $(\tilde{\G}, \tilde{\G}')$ in $\Pin(\E, \b)$ is a dual pair and $\G, \G' \subseteq \SO(\E, \b)$ (here, we don't considerer the dual pairs $(\G, \G') = (\O(\U, b_{\U}), \O(\V, b_{\V}))$, where $\U, \V$ are both real or complex orthogonal vector spaces, see Remark \ref{FinalRemark}). We want to prove that $(\Pi, \S)$ sets up a Howe correspondence for the pair $(\tilde{\G}, \tilde{\G}')$, i.e.
\begin{equation*}
\Pi = \bigoplus\limits_{\lambda \in \hat{\G}} \lambda \otimes \theta(\lambda),
\end{equation*}
where $\theta(\lambda)$ is an irreducible representation of $\tilde{\G}'$.

\begin{rema}

We keep the notations of the previous paragraph. We consider first the action of $\tilde{\G}$ on $\S$ and consider a decomposition of $\S$ as a direct sum of $\tilde{\G}$-isotypic component:
\begin{equation*}
\S = \bigoplus\limits_{\lambda \in \tilde{\G}_{\irr}} \V(\lambda).
\end{equation*}
Because $\tilde{\G}'$ commutes with $\tilde{\G}$, $\tilde{\G'}$ acts on $\V(\lambda)$. In particular, $\tilde{\G} \times \tilde{\G'}$ acts on $\V(\lambda)$ and then,
\begin{equation*}
\Pi = \bigoplus\limits_{\lambda \in \tilde{\G}_{\irr}} \lambda \otimes \theta(\lambda)
\end{equation*}
where $\theta(\lambda)$ is a representation of $\tilde{\G}'$, but non-necessary irreducible. The goal of this section is to prove the irreducibility of the representation $\theta(\lambda)$.

\end{rema}

According to the double commutant theorem (see \cite[Section~4.1]{GOD}), if 
\begin{equation*}
\Comm_{\End(\S)}\left<\tilde{\G}\right> = \left\{X \in \End(\S), \Pi(\tilde{g})X\Pi(\tilde{g})^{-1} = X \thinspace (\forall \tilde{g} \in \tilde{\G})\right\} =  \left<\tilde{\G}'\right>,
\end{equation*}
then, the representation $\theta(\lambda)$ of $\tilde{\G}'$ corresponding to $\lambda$ is irreducible (here, $\left<\tilde{\G}\right>$ is the $\mathbb{C}$-algebra generated by $\Pi(\tilde{g}), \tilde{g} \in \tilde{\G}$). But, using the notations of the previous remark, we get:
\begin{equation*}
\Comm_{\End(\S)}\left<\tilde{\G}\right> = \End(\S)^{(\tilde{\G}, \tilde{\Pi})}.
\end{equation*}
(here we add the representations $\tilde{\Pi}$ in the notations to avoid any confusion). Then, according to Equation \eqref{TransferOfInvariants} and Remark \ref{RemarkInvariants}, we have:
\begin{equation}
\Comm_{\End(\S)}\left<\tilde{\G}\right> = \T_{1} \circ \T^{-1}\left(\Lambda(\E_{\mathbb{C}})^{(\tilde{\G}, \tilde{\rho_{1}})}\right) = \T_{1} \circ \T^{-1}\left(\Lambda(\E_{\mathbb{C}})^{(\G, \pr\tilde{\rho_{1}})}\right).
\label{DoubleCommutantIsomorphism}
\end{equation}
In particular, for the dual pairs we consider, we will determine explicitly $\Lambda(\E)^{\G, \pr\rho_{1}}$. In \cite{HOW2}, R. Howe computed invariants of this form for some complex groups. It will be one of our key points to prove the duality for the spinorial representation. We recall the main ideas in this section (one can also check \cite{WAN}). It turns out that the invariants we need are, indirectly, a sort of consequence of the Schur-Weyl duality. 

Let $\G = \GL(V)$ where $V$ is a finite dimensional complex vector space. We have a natural action of $\G$ on $\V$, which can be extended to an action of $\G$ on $\V^{\otimes d}$ diagonally, i.e.
\begin{equation*}
g.(v_{1} \otimes \ldots \otimes v_{d}) = g(v_{1}) \otimes \ldots \otimes g(v_{d}).
\end{equation*}
Similarly, we have an action of the symmetric group $\mathscr{S}_{d}$ on $\V^{\otimes d}$ given by
\begin{equation*}
\sigma.(v_{1} \otimes \ldots \otimes v_{d}) = v_{\sigma^{-1}(1)} \otimes \ldots \otimes v_{\sigma^{-1}(d)},
\end{equation*}
and this action commutes with the action of $\G$. More particularly, we have the following decomposition:
\begin{equation*}
\V^{\otimes d} = \bigoplus\limits_{\D} \rho^{\D}_{\V} \otimes \sigma^{\D},
\end{equation*}
where $\D$ runs over the diagrams of size $d$ and depth at most $\dim_{\mathbb{C}}(\V)$.

A direct consequence of this duality is a $(\GL(\U), \GL(\V))$-duality on the spaces $\S(\U \otimes \V)$ and $\Lambda(\U \otimes \V)$. Indeed, we have
\begin{eqnarray*}
\Lambda^{d}(\U \otimes \V) & = & \left[(\U \otimes \V)^{d}\right]^{\sgn(d)} = \left[\U^{\otimes d} \otimes \V^{\otimes d}\right]^{\Delta(\sgn(d) \times \sgn(d))} \\
& = & \left[ \left( \sum\limits_{\D} \rho^{\D}_{\U} \otimes \sigma^{\D}\right) \otimes \left( \sum\limits_{\E} \rho^{\E}_{\V} \otimes \sigma^{\E}\right)\right]^{\Delta(\sgn(d) \times \sgn(d))} \\
& = & \sum\limits_{\D} \sum\limits_{\E} \rho^{\D}_{\U} \otimes \rho^{\E}_{\V} \otimes \left(\sigma^{\D} \otimes \sigma^{\E}\right)^{\Delta(\sgn(d) \times \sgn(d))} \\
& = & \sum\limits_{\D} \rho^{\D}_{\U} \otimes \rho^{\D'}_{\V} 
\end{eqnarray*}
where $\sgn(d)$ is the action of $\mathscr{S}_{d}$ on $\V^{\otimes d}$ via the signature, $\D$ runs over the diagrams of size $d$ such that depth at most $\dim_{\mathbb{C}}(\U)$ and $\D'$ is associated to $\D$ as constructed in \cite[Annexe~A.29]{WAN}.

Using that result, one can get the $\GL(\V)$-invariants in the space $\Lambda(\V \otimes \U \oplus \V^{*} \otimes \W)$. Obviously, fixing a basis $\{v_{1}, \ldots, v_{n}\}$ of $\V$, $\{v^{*}_{1}, \ldots, v^{*}_{n}\}$ dual basis of $\V^{*}$, $\{u_{1}, \ldots, u_{m}\}$ of $\U$ and $\{w_{1}, \ldots, w_{l}\}$ of $\W$, the elements $\lambda_{a, b}$ given by
\begin{equation}
\lambda_{a, b} = \sum\limits_{k=1}^{n} v_{k} \otimes u_{a} \wedge v^{*}_{k} \otimes w_{b}
\label{InvariantGL(V)}
\end{equation} 
are $\GL(\V)$-invariants in $\Lambda(\V \otimes \U \oplus \V^{*} \otimes \W)$. According to R. Howe, all the invariants can be obtained using those elements $\lambda_{a, b}$.

\begin{theo}[\cite{HOW2}]

The algebra $\Lambda(\V \otimes \U \oplus \V^{*} \otimes \W)^{\GL(\V)}$ is generated by the degree $2$ elements $\lambda_{a, b}$ given in equation \eqref{InvariantGL(V)}.
\label{HoweGL}

\end{theo}

One of the key part of the proof is to understand the structure of $\Lambda(\V \otimes \U \oplus \V^{*} \otimes \W)^{\GL(\V)}$ as a $\GL(\U) \times \GL(\W)$-module. And this is a direct consequence of the duality we had previously. Indeed,
\begin{eqnarray}
\Lambda(\V \otimes \U \oplus \V^{*} \otimes \W)^{\GL(\V)} & = & \left(\Lambda(\V \otimes \U) \otimes \Lambda(\V^{*} \otimes \W)\right)^{\GL(\V)} = \left(\left(\sum\limits_{\D} \rho^{\D}_{\V} \otimes \rho^{\D}_{\U}\right) \otimes \left(\sum\limits_{\E} \rho^{\E}_{\V} \otimes \rho^{\E}_{\W}\right)\right)^{\GL(\V)} \nonumber \\
           & = & \sum\limits_{\D} \sum\limits_{\E} (\rho^{\D}_{\V} \otimes \rho^{\E}_{\V})^{\GL(\V)} \otimes \rho^{\D}_{\U} \otimes \rho^{\E}_{\W} = \sum\limits_{\D} \rho^{\D}_{\U} \otimes \rho^{\D^{'}}_{\W} \label{GLInvProof}
\end{eqnarray}
We remark that the last equation is exactly the decomposition of $\Lambda(\U \otimes \W)$ as a $\GL(\U) \times \GL(\W)$-module.

The rest of the proof is to justify that the highest weight of the modules appearing in Equation \ref{GLInvProof} can be obtained using the elements $\lambda_{a, b}$ (one can check \cite[Proof~of~Theorem~2.2.2]{HOW2}).

Other consequences of this $(\GL(\U), \GL(\V))$-duality in $\Lambda(\U \otimes \V)$ is the explicit description of the algebras $\Lambda(\U \otimes \V)^{\O(\U)}$ and $\Lambda(\U \otimes \V)^{\Sp(\U)}$.

Indeed, let $\{u_{1}, \ldots, u_{n}\}$ be a basis of $\U$ and $\{v_{1}, \ldots, v_{m}\}$ be a basis of $\V$. Then, 
\begin{equation*}
\eta_{a, b} = \sum\limits_{k = 1}^{n} u_{k} \otimes v_{a} \wedge u_{k} \otimes v_{b}
\end{equation*}
are elements of $\Lambda(\U \otimes \V)^{\O(\U)}$. 

\begin{theo}[\cite{HOW2}]

The elements $\eta_{a, b}$ generate the algebra $\Lambda(\U \otimes \V)^{\O(\U)}$. In particular, the algebra $\Lambda(\U \otimes \V)^{\O(\U)}$ is generated by degree $2$ elements.
\label{HoweO}
\end{theo}

\begin{proof}

See \cite[Theorem~4.3.2]{HOW2}.

\end{proof}

Similarly, let's consider the space $\mathbb{C}^{2n} \otimes \mathbb{C}^{m}$ and let $\Sp(2n, \mathbb{C})$ be the complex symplectic group. We consider a symplectic basis $\{e_{1}, \ldots, e_{n}, f_{1}, \ldots, f_{n}\}$ of $\mathbb{C}^{2n}$ and let $\{x_{1}, \ldots, x_{m}\}$ be a basis of $\mathbb{C}^{m}$. The elements $\gamma_{a, b}$ given by:
\begin{equation*}
\gamma_{a, b} = \sum\limits_{k=1}^{n} e_{k} \otimes x_{a} \wedge f_{k} \otimes x_{b} - f_{k} \otimes x_{a} \wedge e_{k} \otimes x_{b}
\end{equation*}
are in $\Lambda(\mathbb{C}^{2n} \otimes \mathbb{C}^{m})^{\Sp(2n, \mathbb{C})}$. R. Howe proved the following result.

\begin{theo}[\cite{HOW2}]

The elements $\gamma_{a, b}$ generate the algebra $\Lambda(\mathbb{C}^{2n} \otimes \mathbb{C}^{m})^{\Sp(2n, \mathbb{C})}$.
\label{HoweSp}
\end{theo}

\begin{proof}

See \cite[Theorem~3.7.8.2]{HOW2}.

\end{proof}

We can first deduce the duality for the dual pairs in the complex orthogonal group.

\begin{prop}

Let $(\G, \G') \in \O(\E,b)$ one of the complex irreducible reductive dual pair given in Corollary \ref{ClassificationDP} except $(\G = \O(n, \mathbb{C}), \G' = \O(m, \mathbb{C}))$ (see Remark \ref{FinalRemark} for these dual pairs) . Let $(\Pi, \S)$ be the corresponding representation of $\Pin(\E, \b)$. Then, $\S$ sets up a Howe correspondence for the dual pair $(\tilde{\G}, \tilde{\G'})$.
\label{ComplexPairs}
\end{prop}

\begin{proof}

As explained previously, we have to prove that $\Comm_{\End(\S)}\langle\tilde{\G}\rangle = \langle\tilde{\G'}\rangle$ and according to equation \eqref{DoubleCommutantIsomorphism}, one way is to determine $\Lambda(\E)^{\G}$.

If $(\G = \GL(\U), \G' = \GL(\V))$, then $\E = \U \otimes \V \oplus \U^{*} \otimes \V^{*}$ and according to theorem \ref{HoweGL}, we have:
\begin{eqnarray*}
\Lambda(\E)^{\GL(\U)} & = & \Lambda(\U \otimes \V \oplus \U^{*} \otimes \V^{*})^{\GL(\U)} = \left<\left((\U \otimes \V) \otimes (\U \otimes \V)^{*}\right)^{\GL(\U)}\right> \\
                                     & = & \left< \mathfrak{gl}(\U \otimes \V)^{\GL(\U)} \right> = \left<\mathfrak{gl}(\V)\right>
\end{eqnarray*}
where $\left< \cdot\right>$ is the subalgebra of $\Lambda(\U \otimes \V \oplus \U^{*} \otimes \V^{*})$ generated by $\mathfrak{gl}(\V)$. Via the map $\T$ and $\tilde{\gamma}$, $\Lambda(\U \otimes \V \oplus \U^{*} \otimes \V^{*})$ corresponds to the subalgebra of $\End(\S)$ generated by the operators $d\Pi(X), X \in \mathfrak{gl}(\V)$.

If $(\G, \G') = (\Sp(\U, \b_{\U}), \Sp(\V, \b_{\V}))$, then, $\E = \U \otimes \V$ and according to Theorem \ref{HoweSp}, we have:
\begin{eqnarray*}
\Lambda(\E)^{\G} & = & \Lambda(\U \otimes \V)^{\G} = \left< \Lambda^{2}(\U \otimes \V)^{\G}\right> \\
                             & = & \left< \mathfrak{so}(\U \otimes \V)^{\G} \right> = \left<\mathfrak{sp}(\V, \b_{\V})\right>
\end{eqnarray*}                                    

\end{proof}

We now apply the same method for the real orthogonal group. The main difference in that case is that the members of the dual pairs we deal with are real algebraic groups. One of the main point is to complexify these groups. If $\G$ is a real algebraic group, the complexification of $\G$ is a pair $(\G_{\mathbb{C}}, \gamma)$, where $\G$ is a complex Lie group and $\gamma: \G \to \G_{\mathbb{C}}$ is a $\mathbb{R}$-morphism analytic map such that for every complex Lie group $\H$ and every $\mathbb{R}$-morphism $\phi: \G \to \H$, there exists a unique $\mathbb{C}$-analytic map $\psi: \G_{\mathbb{C}} \to \H$ such that $\phi = \psi \circ \gamma$ (one can check \cite[Chapter~3,~6.10]{BOU}).

Let $(\G, \G')$ be an irreducible dual pair in $\O(\E, \b)$. Let $\O(\E_{\mathbb{C}}, b_{\mathbb{C}})$ be the corresponding complexification and $(\Pi, \S)$ be the corresponding spinorial representation. Using Equation \eqref{TransferOfInvariants}, we identify the algebra $\Comm_{\End(\S)}(\tilde{\G})$ with the space $\Lambda(\E_{\mathbb{C}})^{\G}$, and because $\E_{\mathbb{C}}$ is a complex vector space, we have:
\begin{equation*}
\Lambda(\E_{\mathbb{C}})^{\G} = \Lambda(\E_{\mathbb{C}})^{\G_{\mathbb{C}}},
\end{equation*} 
where the action of $\G_{\mathbb{C}}$ is uniquely defined by definition of the complexification.

We recall quickly some complexifications. Let $\G = \GL(n, \mathbb{C})$. As a real Lie group, we have:
\begin{equation*}
\GL(n, \mathbb{C}) = \left\{g \in \GL(2n, \mathbb{R}), gJ = Jg\right\}, \qquad J = \begin{pmatrix} 0 & \Id_{n} \\ -\Id_{n} & 0 \end{pmatrix}.
\end{equation*}
We have $J = PDP^{-1}$, where $D = \diag(i\Id_{n}, -i\Id_{n})$. The complexification of $\GL(n, \mathbb{C})$ is given by:
\begin{equation*}
\GL(n, \mathbb{C})_{\mathbb{C}} = \left\{g \in \GL(2n, \mathbb{C}), gJ = Jg\right\} = \left\{g \in \GL(2n, \mathbb{C}), P^{-1}gPJ = JP^{-1}gP\right\}
\end{equation*}
In a suitable basis, we get that $P^{-1}gP$ is diagonal by blocks of the form $\diag(X_{1}, X_{2})$, where $X_{1}, X_{2} \in \GL(n, \mathbb{C})$. Then,
\begin{equation*}
\GL(n, \mathbb{C})_{\mathbb{C}} = \GL(n, \mathbb{C}) \times \GL(n, \mathbb{C}) \qquad \gamma(g) = (g, \bar{g}).
\end{equation*}
More generally, we have:
\begin{equation*}
\Sp(\V, \b)_{\mathbb{C}} = \Sp(\V, \b) \times \Sp(\V, \b) \qquad \O(\V, b)_{\mathbb{C}} = \O(\V, \b) \times \O(\V, b).
\end{equation*}

\begin{rema}

\begin{enumerate}
\item Let $\V$ be a complex vector space and $\G = \GL(\V), \Sp(\V)$ or $\O(\V)$. We denote by $\Pi: \G \to \GL(V)$ be the natural action of $\G$ on $\V$. The extension of $\Pi$ to $\G_{\mathbb{C}}$, denoted by $\Pi_{\mathbb{C}}$, is given by:
\begin{equation*}
\Pi_{\mathbb{C}}(g_{1}, g_{2}) = \Pi(g_{1}).
\end{equation*}
\item Let $\V$ be a complex vector space. We denote by $\bar{\V}$ the space complex vector space such that
\begin{itemize}
\item $\V = \bar{\V}$,
\item For each $\lambda \in \mathbb{C}$ and $\bar{v} \in \bar{\V}$, the vector $\lambda \bar{v}$ is equal to the vector $\bar{\lambda}v$.
\end{itemize}
We denote by $\bar{\Pi}$ the natural action of $\G$ on $\bar{\V}$ given by:
\begin{equation*}
\bar{\Pi}(g)(\bar{v}) = \overline{\Pi(g)(v)}.
\end{equation*}
We denote by $\bar{\Pi}_{\mathbb{C}}$ the extension of $\G$ to $\G_{\mathbb{C}}$. Then,
\begin{equation*}
\bar{\Pi}_{\mathbb{C}}(g_{1}, g_{2}) = \bar{\Pi}(g_{2}).
\end{equation*}
\end{enumerate}

\end{rema}

We denote by $\mathbb{H} = \mathbb{R}[i] \oplus j\mathbb{R}[i]$ the set of quaternionic numbers. We have a natural embedding of $\M(n, \mathbb{H})$ in $\M(2n, \mathbb{R}[i])$ given by:
\begin{equation}
\M(n, \mathbb{H}) \ni A + jB \to \begin{pmatrix} A & -\bar{B} \\ B & \bar{A} \end{pmatrix} \in \M(2n, \mathbb{R}[i]).
\label{EmbeddingQuaternions}
\end{equation}
We also denote by $\GL(n, \mathbb{H})$ its image in $\M(2n, \mathbb{C})$. More precisely, we have:
\begin{equation*}
\GL(n, \mathbb{H}) = \left\{g \in \GL(2n, \mathbb{C}), gJ = J\bar{g}\right\}.
\end{equation*}
and the complexification of $\GL(n, \mathbb{H})$ is $\GL(2n, \mathbb{C})$.

Similarly, we have
\begin{equation*}
\Sp(p, q, \mathbb{H})_{\mathbb{C}} = \Sp(2(p+q), \mathbb{C}) \qquad \O^{*}(n, \mathbb{H})_{\mathbb{C}} = \O(2n, \mathbb{C}),
\end{equation*}
where we identify $\Sp(p, q, \mathbb{H})$ and $\O^{*}(n, \mathbb{H})$ with their images in $\M(2n, \mathbb{C})$.

We are now able to prove the duality for the real orthogonal group. We divide the proof of the duality in three differents propositions.

\begin{prop}

Let $(\G, \G') = (\GL(\U), \GL(\V))$ or $(\Sp(\U, \b_{\U}), \Sp(\V, \b_{\V})$ be a dual pair in their corresponding orthogonal group $\O(\E, \b)$, where $\U$ and $\V$ are real vector spaces. Then, the spinorial representation of $\Pin(\E, \b)$ sets up a Howe correspondence for the pair $(\tilde{\G}, \tilde{\G}')$.

\end{prop}

\begin{proof}

The proof is similar to the one of Proposition \ref{ComplexPairs}.

\end{proof}

\begin{prop}

Let $(\G, \G') = (\GL(\U), \GL(\V)), (\U(\U, \b_{\U}), \U(\V, \b_{\V})), (\Sp(\U, \b_{\U}), \Sp(\V, \b_{\V}))$ be a dual pair in their corresponding real orthogonal group $\O(\E, \b)$, where $\U$ and $\V$ are vector spaces over $\mathbb{C}$. Then, the spinorial representation of $\Pin(\E, \b)$ sets up a Howe correspondence for the pair $(\tilde{\G}, \tilde{\G}')$.

\end{prop}

\begin{proof}

We start with the pair $(\G, \G') = (\U(\U, \b_{\U}), \U(\V, \b_{\V})) \subseteq \O((\U \otimes_{\mathbb{C}} \V)_{\mathbb{R}}, \b)$, where $\b = \Re(\b_{\U} \otimes \b_{\V})$. In this case, we have $\E = (\U \otimes_{\mathbb{C}} \V)_{\mathbb{R}}$ and according to \cite[Section~4.1]{TRIO}, we have 
\begin{equation*}
E_{\mathbb{C}} = \U \otimes_{\mathbb{C}} \V \oplus \tilde{\U} \otimes_{\mathbb{C}} \tilde{\V}.
\end{equation*}
According to Theorem \ref{HoweGL}, 
\begin{eqnarray*}
\Lambda(\E_{\mathbb{C}})^{\G} & = & \Lambda(\U \otimes_{\mathbb{C}} \V \oplus \tilde{\U} \otimes_{\mathbb{C}} \tilde{\V})^{\GL(\U)} = \Lambda(\U \otimes_{\mathbb{C}} \V \oplus (\bar{\U})^{*} \otimes_{\mathbb{C}} (\bar{\V})^{*})^{\GL(\U)} \\
                                                   & = & \left< \left(\U \otimes_{\mathbb{C}} \V \oplus (\bar{\U})^{*} \otimes_{\mathbb{C}} (\bar{\V})^{*}\right)^{\GL(\U)}\right> \\
                                                   & = & \left<\mathfrak{gl}(\U \otimes \V)^{\GL(\U)}\right> = \left<\mathfrak{gl}(\V)\right> = \left<\mathfrak{u}(\V)\right>
\end{eqnarray*}   
where $\left<\mathfrak{u}(\V)\right>$ is identified via the map $\T$ and $\tilde{\gamma}$ with the $\mathbb{C}$-algebra generated by the operators $d\Pi(X), X \in \mathfrak{u}(\V)$.

Let's now consider the dual pair $(\GL(\U), \GL(\V)) \in \O((\U \otimes \V)_{\mathbb{R}} \oplus (\U \otimes \V)^{*}_{\mathbb{R}}, \b)$. In this case, we have:
\begin{equation*}
\E_{\mathbb{C}} = \U \otimes \V \oplus \bar{\U} \otimes \bar{\V} \oplus \U^{*} \otimes \V^{*} \oplus \bar{\U}^{*} \otimes \bar{\V}^{*}                                         
\end{equation*}
and then
\begin{eqnarray*}
\Lambda(\E_{\mathbb{C}})^{\G} & = & \Lambda(\U \otimes \V \oplus \bar{\U} \otimes \bar{\V} \oplus \U^{*} \otimes \V^{*} \oplus \bar{\U}^{*} \otimes \bar{\V}^{*} )^{\GL(\U) \times \GL(\U)} \\
                                                  & = & \Lambda(\U \otimes \V  \oplus \U^{*} \otimes \V^{*})^{\GL(\U)} \otimes \Lambda( \bar{\U} \otimes \bar{\V} \oplus \bar{\U}^{*} \otimes \bar{\V}^{*} )^{\GL(\U)} \\
                                                   & = & \left< (\U \otimes \V  \oplus \U^{*} \otimes \V^{*})^{\GL(\U)}\right> \otimes \left<(\bar{\U} \otimes \bar{\V} \oplus \bar{\U}^{*} \otimes \bar{\V}^{*})^{\GL(\U)}\right> \text{using Theorem \ref{HoweGL}} \\
                                                   & = & \left<\mathfrak{gl}(\U \otimes \V)^{\GL(\U)}\right> \otimes \left<\mathfrak{gl}(\bar{\U} \otimes \bar{\V})^{\GL(\U)}\right>  \\
                                                   & = & \left< \mathfrak{gl}(\V)_{\mathbb{C}}\right> = \left< \mathfrak{gl}(\V)\right>.
\end{eqnarray*} 
We finish this proof with the dual pair $(\Sp(\U, \b_{\U}), \Sp(\V, \b_{\V}))$ in $\O(\E, \b)$, where $\E = (\U \otimes \V)_{\mathbb{R}}$. Then, we have:
\begin{equation*}
\E_{\mathbb{C}} = \U \otimes \V \oplus \bar{\U} \otimes \bar{\V}
\end{equation*}
and 
\begin{eqnarray*}
\Lambda(\E_{\mathbb{C}})^{\G} & = & \Lambda(\U \otimes \V \oplus \bar{\U} \otimes \bar{\V})^{\G \times \G} = \Lambda(\U \otimes \V)^{\G} \otimes \Lambda(\bar{\U} \otimes \bar{\V})^{\G} \\
                             & = & \left<\Lambda^{2}(\U \otimes \V)^{\G}\right> \otimes \left<\Lambda^{2}(\bar{\U} \otimes \bar{\V})^{\G}\right>\text{using Theorem \ref{HoweO}}  \\
                             & = & \left< \mathfrak{so}(\U \otimes \V)^{\G}\right> \otimes \left< \mathfrak{so}(\bar{\U} \otimes \bar{\V})^{\G}\right> \\
                             & = & \left<\mathfrak{g}'_{\mathbb{C}}\right> = \left<\mathfrak{g}'\right>.
\end{eqnarray*}                            

\end{proof}

\begin{prop}

Let $(\G, \G') = (\GL(\U), \GL(\V)), (\Sp(\U, \b_{\U}), \Sp(\V, \b_{\V}))$ or $(\O^{*}(\U, \b_{\U}), \O^{*}(\V, \b_{\V}))$ be a dual pair in their corresponding orthogonal group $\O(\E, \b)$, where $\U$ and $\V$ are vector spaces over $\mathbb{H}$. Then, the spinorial representation of $\Pin(\E, \b)$ sets up a Howe correspondence for the pair $(\tilde{\G}, \tilde{\G}')$.

\end{prop}

\begin{proof}

We start with the dual pair $(\G = \GL(\U), \G' = \GL(\V)) \subseteq \O(\E, \b)$. In this case, we have $\E = \U \otimes_{\mathbb{H}} \V \oplus \U^{*} \otimes_{\mathbb{H}} \V^{*}$ and then, $\E_{\mathbb{C}} = \U_{1} \otimes V_{1} \oplus \U^{*}_{1} \otimes V^{*}_{1}$, where $\U_{1}$ and $V_{1}$ are the complex space associated to $\U$ and $\V$ respectively. Then, we have:
\begin{eqnarray*}
\Lambda(\E_{\mathbb{C}})^{\G} & = & \Lambda(\E_{\mathbb{C}})^{\G_{\mathbb{C}}} = \Lambda(\U_{1} \otimes V_{1} \oplus \U^{*}_{1} \otimes V^{*}_{1})^{\GL(\U_{1})} \\
                                                  & = & \left< (\U_{1} \otimes V_{1} \oplus \U^{*}_{1} \otimes V^{*}_{1})^{\GL(\U_{1})}\right> =\left<\mathfrak{gl}(\U_{1} \otimes V_{1})^{\GL(U_{1})}\right> \\
                                                  & = & \left< \mathfrak{gl}(V_{1}) \right> = \left< \mathfrak{g}'\right>.
\end{eqnarray*}
Similarly, let $(\G, \G') = (\Sp(\U), \Sp(\V))$ in $\O(\E, \b)$ with $\E = \U \otimes_{\mathbb{H}} \V$. Then, $\E_{\mathbb{C}} = \U_{1} \otimes \V_{1}$ and    
\begin{eqnarray*}                                            
\Lambda(\E_{\mathbb{C}})^{\G} & = & \Lambda(\E_{\mathbb{C}})^{\G_{\mathbb{C}}} = \Lambda(\U_{1} \otimes V_{1})^{\Sp(\U_{1})} \\
                                                  & = & \left< \mathfrak{so}(\U_{1} \otimes V_{1})^{\Sp(\U_{1})}\right> \text{using Theorem \ref{HoweSp}} \\
                                                  & = & \left<\mathfrak{sp}(\V_{1})\right> \\
                                                  & = & \left< \mathfrak{g}' \right>.
\end{eqnarray*}
Finally, for $(\G, \G') = (\O^{*}(\U), \O^{*}(\V))$ in $\O(\E, \b)$ with $\E = \U \otimes_{\mathbb{H}} \V$we have $\E_{\mathbb{C}} = \U_{1} \otimes \V_{1}$ and    
\begin{eqnarray*}                                            
\Lambda(\E_{\mathbb{C}})^{\G} & = & \Lambda(\E_{\mathbb{C}})^{\G_{\mathbb{C}}} = \Lambda(\U_{1} \otimes V_{1})^{\O(\U_{1})} \\
                                                  & = & \left< \mathfrak{so}(\U_{1} \otimes V_{1})^{\O(\U_{1})}\right> = \left<\mathfrak{o}(\V_{1})\right> \\
                                                  & = & \left< \mathfrak{g}' \right>.
\end{eqnarray*}

\end{proof}

\begin{rema}
\label{FinalRemark}
For complex dual pairs $(\O(n,\mathbb{C}),\O(m,\mathbb{C}))$ in $\O(nm,\mathbb{C})$ with $n,m>2$, one can understand the duality of lifts directly from the real case done in \cite{SLU} because $\O(N,\mathbb{R})$ is a maximal compact in $\O(N,\mathbb{C})$ in general. 

In \cite[Corollary~3.5]{SLU}, Slupinski proved that $(\tilde{\O(n,\mathbb{R})},\tilde{\O(m,\mathbb{R})})$ is a dual pair in $\Pin(nm,\mathbb{R})$ if both $n$ and $m$ are odd. If $n$ or $m$ is even and $n,m>2$ then 
$(\tilde{\SO(n,\mathbb{R})},\tilde{\O(m,\mathbb{R})})$ and $(\tilde{\O(n,\mathbb{R})},\tilde{\SO(m,\mathbb{R})})$ are both dual pairs in $\Pin(nm,\mathbb{R})$. Slupinski also computed (see \cite[Corollary~2.13]{SLU}) the isomorphism classes of these lifts. Because of Lemma \ref{redcomp}, we have the same results for complex pairs changing $\mathbb{R}$ to $\mathbb{C}$. 

Finally, Slupinski proved that for these real dual pairs, the spinorial representation sets up a Howe correspondence. Because there is a bijection between finite dimensional representations of a complex simple group and its compact maximal, the corresponding result for the complex pair is true as well. 

\end{rema}

\end{document}